\documentclass[11pt,a4paper]{article}
\usepackage[utf8]{inputenc}
\usepackage[english]{babel}
\usepackage{amsmath}
\usepackage{amsthm}
\usepackage{enumerate}
\usepackage{amsfonts}
\usepackage{cancel}
\usepackage{amssymb}
\usepackage{makeidx}
\usepackage{graphicx}
\usepackage{lmodern}
\usepackage{indentfirst}
\usepackage{xcolor}

\newtheorem{defi}{Definition}[section]
\newtheorem{teo}[defi]{Theorem}
\newtheorem{lem}[defi]{Lemma}

\newtheorem{rem}[defi]{Remark}

\numberwithin{equation}{section}

\newcommand{\R}{\mathbb{R}}
\newcommand{\Sp}{\mathbb{S}}
\newcommand{\Hy}{\mathbb{H}}
\newcommand{\Q}{\mathbb{Q}}
\newcommand{\E}{\mathbb{E}}
\newcommand{\e}{\epsilon}
\newcommand{\pie}{\langle}
\newcommand{\pid}{\rangle}

\newcommand{\WP}{\R\times_{\rho}\Q^n_{\epsilon}}
\newcommand{\IWP}{I\times_{\rho}\Q^n_{\epsilon}}
\newcommand{\co}{\colon}
\newcommand{\Span}{\text{span}}
\newcommand{\dt}{\partial_t}
\newcommand{\al}{\alpha}
\newcommand{\np}{\nabla^\perp}
\newcommand{\nb}{\overline\nabla}
\newcommand{\nt}{\widetilde\nabla}
\newcommand{\pl}{\langle}
\newcommand{\pr}{\rangle}

\title{Surfaces with parallel mean curvature in warped product spaces}

\author{Fernando Manfio\footnote{The first author is supported by FAPESP, grant 2022/16097-2. The second author was supported by CAPES, grant 88887.342501/2019-00.}, Verônica Reis and Feliciano Vitório}
\date{}

\newcommand{\Addresses}{{
\bigskip 
\footnotesize
University of S\~ao Paulo, Brazil \\
\textit{E-mail address:} \texttt{manfio@icmc.usp.br, reissantanaveronica@gmail.com}
\vspace{.2cm} \\
Federal University of Alagoas, Brazil \\
\textit{E-mail address:} \texttt{feliciano@pos.mat.ufal.br}
}}

\begin{document}	

\maketitle

\begin{abstract}
In this work, we obtain a geometric description of surfaces $M^2$ of arbitrary codimension in the warped product $\R\times_\rho\Q^n_\e$, with parallel mean curvature vector field in the normal connection, extending a result by Alencar-do Carmo-Tribuzy \cite{AlCaTr}.
\end{abstract}


\section{Introduction}

A classical result concerning surfaces $M^2$ in Euclidean space $\mathbb{R}^3$, due to Hopf \cite{Ho}, states that any surface homeomorphic to a sphere with constant mean curvature is, in fact, isometric to a round sphere. This theorem was later extended by Chern \cite{Ch} to other space forms and, more recently, to surfaces in $3$-dimensional homogeneous manifolds with a $4$-dimensional isometry group by Abresch and Rosenberg \cite{AbRo}.

A natural class of ambient spaces for which Hopf-type theorems hold is that of product spaces $\Q^n_\e\times\R$, where $\Q^n_\e$ denotes either the unit sphere $\Sp^n$ or the hyperbolic space $\Hy^n$, according as $\e=1$ or $\e=-1$, respectively. A pioneering work in this direction was conducted by Alencar-do Carmo-Tribuzy \cite{AlCaTr}, who studied surfaces with parallel mean curvature vector in $\Q^n_\e\times\R$. They proved that either $H$ is an umbilic direction or the immersion admits a reduction of codimension to three.  

The study of surfaces and, more generally, submanifolds with parallel mean curvature vector field has long been a subject of interest in differential geometry. These submanifolds generalize both minimal and constant mean curvature surfaces by requiring the mean curvature vector $H$ to be parallel with respect to the normal connection. As demonstrated by Chen \cite{Chen}, Yau \cite{Yau}, and Aliás-Dajczer \cite{AlDa}, for example, this condition imposes significant rigidity and geometric constraints on the immersion.

In this work, we investigate surfaces with parallel mean curvature vector field in higher-dimensional warped product spaces $\R\times_\rho\Q^n_\e$. Following the approach in \cite{AlCaTr}, our main result does not require global assumptions on the surface $M^2$. We show that either $H$ is an umbilical direction or the immersion admits a reduction of codimension. Specifically, we establish the following:

\begin{teo} \label{teo:main}
Let $f\co M^2\to\WP$, $n\geq 5$, be a surface with nonzero parallel mean curvature vector field. Then, one of the following possibilities holds:
\begin{enumerate}
\item[(i)] $f$ is a minimal surface of a umbilical hypersurface of a slice $\{t\}\times_{\rho}\Q^{n}_{\epsilon}$.
\item[(ii)] $f$ is a surface with constant mean curvature in a three-dimensional umbilical or totally geodesic submanifold of a slice $\{t\}\times_{\rho}\Q^{n}_{\epsilon}$.
\item[(iii)] $f(M^2)$ lies in a totally geodesic submanifold $\R\times_{\rho}\Q^{m}_{\epsilon}$, $m\leq 4$, of $\WP$.
\end{enumerate}
\end{teo}

A key step in the proof of Theorem \ref{teo:main} is a codimension reduction result for isometric immersions into $\R\times_\rho\Q^n_\e$ (cf. Section \ref{sec:CodimReduc}), which generalizes the results of \cite{MeTo} for the product space $\Q^n_\e\times\R$.

\section{Preliminaries}

Let $\Q^n_\e$ denote either the sphere or the hyperbolic space of dimension $n$, according to whether $\e=1$ or $\e=-1$, respectively.
We also denote by $\E^{n+2}$ either the Euclidean space or the Lorentzian space of dimension $n+2$, according as $\e=1$ or $\e=-1$, respectively. If $(x_1,\ldots,x_{n+2})$ are the standard coordinates on $\E^{n+2}$ with respect to which the flat metric is written as
\[
ds^2 = \e dx_1^2+dx_2^2+\ldots+dx_{n+2}^2,
\]
we will regard $\E^{n+1}$ as
\[
\E^{n+1} = \{(x_1,\ldots,x_{n+2})\in\E^{n+2}:x_{n+2}=0\},
\]
and the quadric $\Q^n_\e$ regarded as
\[
\Q^n_\e = \{(x_1,\ldots,x_{n+1})\in\E^{n+1}:\e x_1^2+x_2^2+\ldots
+x_{n+1}^2=\e\},
\]
with $x_1>0$ if $\e=-1$. 

Given an open interval $I\subset\R$ and a smooth positive function $\rho\co I\to\R$, let us consider the warped product $\IWP$ endowed with the warped product metric
\[
\pie,\pid = \pi_1^\ast dt^2 + (\rho\circ\pi_1)^2\pi_2^\ast
\pie,\pid_{\Q^n_\e},
\]
where $\pi_1\co\IWP\to I$ and $\pi_2\co\IWP\to\Q^n_\e$ denote the canonical projections. We will consider the warped product manifold $\IWP$ as a rotational hypersurface of $\E^{n+2}$ parametrized by
\begin{equation} \label{eq:inclusion_i}
i(t,x) = (\rho(t)x,h(t)),    
\end{equation}
where $h\co I\to\R$ is a smooth function satisfying $\rho'(t)^2+\e h'(t)^2=1$ and $h'(t)>0$, for all $t\in I$ and $x\in\Q^n_\e$. Note that the warped product $\IWP$ admits a distinguished unit vector field
\[
\dt=(\rho'(t)x,h'(t)),
\]
which is tangent to the first factor $I$, i.e., $\dt$ is a horizontal vector field along $\IWP$. If $\nb$ denotes the Levi-Civita connection on $\IWP$, one has
\begin{equation} \label{eq:Nabladelt}
\nb_Z\dt = \frac{\rho'(t)}{\rho(t)}\left(Z-\pie Z,
\dt\pid\dt\right),
\end{equation}
for every $Z\in T(\IWP)$.

\vspace{.2cm}

Given an isometric immersion $f\co M^m\to\WP$, a tangent vector field $T$ on $M^m$ and a normal vector field $\eta$ along $f$ are defined by
\begin{equation} \label{eq:eta}
\partial_t= f_{\ast} T+\eta.    
\end{equation}
Using \eqref{eq:Nabladelt}, we obtain by differentiating \eqref{eq:eta} that
\begin{eqnarray} \label{eq:derT}
\nabla_XT - A^f_\eta X = \frac{\rho'(t)}{\rho(t)}
\left(X-\langle X,T\rangle T\right)
\end{eqnarray}
and
\begin{eqnarray} \label{eq:dereta}
\al_f(X,T) + \nabla^{\perp}_X\eta = -\frac{\rho'(t)}{\rho(t)}
\langle X,T\rangle\eta,
\end{eqnarray}
for all $X\in TM$, where $\alpha_f$ and $\np$ denote the second fundamental form and the normal connection of $f$, respectively. Here $A^f_\eta$ stands for the shape operator of $f$ in the
direction $\eta$, given by
\[
\pie A^f_\eta X,Y\pid = \pie\al_f(X,Y,\eta\pid,
\]
for all $X,Y\in TM$.

The Gauss, Codazzi, and Ricci equations for $f$ are, respectively
\begin{eqnarray} \label{eq:Gauss}
\begin{aligned}
R(X,Y)Z \ = \ & \lambda(t)(X\wedge Y)Z + A^f_{\alpha_f(Y,Z)}X-A^f_{\alpha_f(X,Z)}Y \\
& + \mu(t)\big(\pie X,T\pid Y\wedge T - \pie Y,T\pid X\wedge T\big)Z,
\end{aligned}
\end{eqnarray}
\begin{eqnarray} \label{eq:Codazzi}
(\np_X\alpha_f)(Y,Z) - (\np_Y\alpha_f)(X,Z) = \mu(t)
\pie(Y\wedge X)Z,T\pid\eta
\end{eqnarray}
and
\begin{eqnarray} \label{eq:Ricci}
R^\perp(X,Y)\xi = \alpha_f(X,A^f_\xi Y)-\alpha_f(A^f_\xi X,Y),
\end{eqnarray}
where
\begin{equation} \label{eq:lambdamu}
\lambda(t) = \frac{\e-\rho'(t)^2}{\rho(t)^2}
\quad\text{e}\quad
\mu(t) = \frac{\rho''(t)\rho(t)-\rho'(t)^2+\e}{\rho(t)^2}.
\end{equation}
Equation \eqref{eq:Codazzi} can also be written as
\begin{equation} \label{eq:Codazzi2}
(\nabla_XA^f)(Y,\xi) - (\nabla_YA^f)(X,\xi) = \mu(t)\pl\eta,\xi\pr(X\wedge Y)T,
\end{equation}
where
\[
(X\wedge Y)T = \pl Y,T\pr X-\pl X,T\pr Y.
\]

Although this will not be used in this work, it is worth mentioning that equations \eqref{eq:derT}--\eqref{eq:Ricci} completely determine an isometric immersion $f\co M^m\to\IWP$ up to isometries of $\IWP$ (see \cite[Theorem 5.1]{FiVi}).

\begin{rem}
Recall that in a warped product $\IWP$, with $n\geq2$, the sectional curvature along a plane tangent to $\Q^n_\e$ is given by $(\e-\rho'(t)^2)/\rho(t)^2$. Meanwhile, the sectional curvature along a plane spanned by unit vector $\dt$ and $X$ tangent to $I$ and $\Q^n_\e$, respectively, is $-\rho''(t)/\rho(t)$. Therefore, $\IWP$ has constant sectional curvature equal to $c$ if and only if
\[
\rho'(t)^2+c\rho(t)^2=\e.
\]
Note that this implies $-\rho''(t)/\rho(t)=c$, or equivalently,
\[
\rho''(t)\rho(t)-\rho'(t)^2+\e=0
\]
(see \cite[Proposition 4.6]{MaTo} for further details). Since we assume that our ambient space $\IWP$ does not have constant sectional curvature, we may suppose that the function $\mu(t)$, defined in \eqref{eq:lambdamu}, is nowhere vanishing. 
\end{rem}


We will finish this section by relating the second fundamental forms and normal connections of $f$ and $\tilde f=i\circ f$, where $i$ is the inclusion given in \eqref{eq:inclusion_i}. The hypersurface $i\co\IWP\to\E^{n+2}$ admits a unit normal vector field given by
\[
\hat N_t = (h'(t)x,-\e\rho'(t)).
\]
Thus,
\[
\widetilde\nabla_Z\hat N_t = \frac{h'(t)}{\rho(t)}Z - 
\left(\frac{h'(t)}{\rho(t)} + \e\frac{\rho''(t)}{h'(t)}
\right)\langle Z,\dt\rangle\dt,
\]
for every $Z\in T(\IWP)$, where $\widetilde\nabla$ is the derivative in $\E^{n+2}$. Hence, the shape operator of $i$ in the direction of $\hat N_t$ is given by
\begin{equation} \label{eq:shapeA_i}
A^i_{\hat N_t}Z = -\frac{h'(t)}{\rho(t)}Z +
\left(\frac{h'(t)}{\rho(t)} + \e\frac{\rho''(t)}{h'(t)}
\right)\langle Z,\dt\rangle\dt.
\end{equation}

The normal spaces of $f$ and $\widetilde{f}$ are related by
\[
TM^\perp_{\tilde f} = i_\ast TM^\perp_f\oplus span\{N_t\},
\]
where $N_t$ denotes the restriction of $\hat N_t$ along $f$. Given a normal vector field $\xi\in TM^\perp_f$, we obtain from \eqref{eq:shapeA_i} that 

\begin{eqnarray} \label{eq:shapes}
\begin{aligned}
\nt_Xi_\ast\xi &= i_\ast\nb_X\xi + \alpha_i(f_\ast X,\xi) \\
& = -\tilde{f}_\ast A^f_{\xi}X + i_\ast\np_X\xi+ \alpha_i(f_\ast X,\xi)
\end{aligned}
\end{eqnarray}
Therefore, from \eqref{eq:shapes} and using Weingarten formula, one has
\begin{eqnarray} \label{eq:A^fA^tf}
A^{\tilde f}_{i_{\ast}\xi}=A^f_{\xi}
\end{eqnarray}
and
\begin{eqnarray} \label{eq:nabtilperp}
\widetilde{\nabla}^{\perp}_{X} i_{\ast} \xi=  i_\ast\np_X\xi+ \e\left(\frac{h'(t)}{\rho(t)}+\e\frac{\rho''(t)}{h'(t)}\right)\langle X, T\rangle \langle\eta,\xi\rangle N_t,
\end{eqnarray}
for every $\xi\in TM^\perp_f$, where $\widetilde{\nabla}^{\perp}$ is the normal connection of $\tilde{f}$.

\section{Reduction of codimension} \label{sec:CodimReduc}

In order to study surfaces with parallel mean curvature vector field in the warped product $\WP$, we need the following result on reduction of codimension of isometric immersions into $\WP$. We say that an isometric immersion $f\co M^m\to\WP$ reduces codimension to $l$ if $f(M^m)$ is contained in a totally geodesic submanifold $\R\times_\rho\Q^{m+l-1}_\e$ of $\WP$.

\vspace{.2cm}

We denote by $N_1(x)$ the {\em first normal space} of $f$ at $x$, i.e., the subspace of $T_xM^\perp$ spanned by its second fundamental form. A condition, due to Erbacher \cite{Er}, for a submanifold of space forms to reduce codimension is that its first normal spaces form a parallel subbundle of the normal bundle. The following result, proved initially for submanifolds in $\mathbb{Q}^n_c \times \mathbb{R}$ in \cite{MeTo}, is a version of the aforementioned result of Erbacher for submanifolds in warped products.

\begin{lem} \label{lem:RedCodim}
Let $f\co M^m\to\WP$ be an isometric immersion and assume that $L:= N_{1}+\Span\{\eta\}$ is a subbundle of $TM^\perp_f$ of rank $l<n+1-m$ and that $\nabla^{\perp}N_1\subset L$. Then f reduces codimension to $l$.
\end{lem}
\begin{proof}
It follows from \eqref{eq:dereta} that $\np_X\eta\in L$, for every $X\in TM$. Since we have $\np N_1\subset L$ by assumption, it follows that $L$ is a parallel subbundle of $TM^\perp_f$.
Let $L^\perp$ denote the orthogonal complement of $L$ in $TM^\perp_f$. Given $\xi\in L^\perp=N_1^\perp\cap\{\eta\}^\perp$, from \eqref{eq:nabtilperp} and the fact that $L$ is a parallel subbundle of $TM^\perp_f$, we obtain
\[
\widetilde{\nabla}_X^\perp i_{\ast}\xi = 
i_{\ast}\nabla^{\perp}_{X} \xi \in i_{\ast}L^{\perp},
\]
where $\widetilde{\nabla}$ denotes the derivative of $\E^{n+2}$. This shows that $i_{\ast}L^{\perp}$ is a parallel subbundle of $TM^\perp_{\tilde f}$. On the other hand, it follows from \eqref{eq:A^fA^tf} that
\[
i_{\ast}L^{\perp}\subset\widetilde{N}_{1}^{\perp},
\]
where $\widetilde{N}_1(x)$ is the first normal space of $\tilde{f}$ at $x\in M^m$. We obtain from the Weingarten formula for $\tilde{f}$ that
\[
\widetilde{\nabla}_{X}i_{\ast}\xi = -\tilde{f}_{\ast}A_{ i_{\ast}\xi}^{\tilde{f}}X + \widetilde{\nabla}_{X}^{\perp}  i_{\ast}\xi=\widetilde{\nabla}_{X}^{\perp}  i_{\ast}\xi \in i_{\ast}L^{\perp},
\]
This implies that $i_{\ast}L^{\perp}$ is a constant subspace of $\E^{n+2}$, which is orthogonal to $\dt$. 
Denote by $K$ the orthogonal complement of $i_{\ast}L^{\perp}$ in $\E^{n+2}$. For any fixed point $x_0\in M^m$, we have
\[
\tilde f(M)\subset\tilde{f}(x_0)+K.
\]
Since $K$ contains $\dt$ and $N_t(x_0)$, it also contains the position vector $\tilde{f}(x_0)$. Thus, we have $\tilde{f}(x_0)+K=K$. We conclude that
\[
\tilde{f}(M)\subset(\R\times_{\rho}\Q^{n}_{\epsilon})\cap K
= \R\times_{\rho}\Q^{m+l-1}_{\epsilon},
\]
and this completes the proof.
\end{proof}

A necessary and sufficient condition for parallelism of the first normal bundle of a submanifold of a space form in terms of its normal curvature tensor $R^\perp$ and mean curvature vector field $H$ was obtained by Dajczer \cite{Da}, whose proof can be adapted to yield the following result for submanifolds of $\WP$. The corresponding result for submanifolds in $\Q^n_\e\times\R$ was obtained in \cite[Theorem 1.7]{MeTo}.

\begin{teo} \label{teo:redCodim}
Let $f\co M^m\to\WP$ be an isometric immersion and assume that $L:= N_{1}+span\{\eta\}$ is a subbundle of $TM^\perp_f$ of rank $l<n+1-m$. Then $\nabla^{\perp}N_1 \subset L$ if and only if the following two conditions hold:
\begin{enumerate}
\item[(i)] $\nabla^{\perp}R^{\perp}\vert_{L^{\perp}}=0$,
\item[(ii)] $\nabla^{\perp}H \in L$.
\end{enumerate}
\end{teo}
\begin{proof}
Assuming that $\nabla^{\perp}N_1\subset L$, condition $(ii)$ is trivially satisfied, since $H\in N_1$. To prove $(i)$, first notice that for $\xi\in N_{1}^\perp$, the Ricci equation gives
\[
R^\perp(X,Y)\xi = \alpha_f(X,A^f_\xi Y) - \alpha_f(A^f_\xi X,Y)=0.
\]
Given $\xi\in L^\perp$, we have that $\xi\in N_1^\perp$. Moreover, by our assumption, one has $\nabla_Z^\perp\xi\in N_{ 1}^{\perp}$, for all $Z\in TM$. Thus, 
\begin{eqnarray*} 
(\nabla_{Z}R^{\perp})(X,Y,\xi) & = & \nabla_ZR^\perp(X,Y)\xi- R^{\perp}( \nabla_{Z}X,Y)\xi - R^{\perp}(X,\nabla_{Z}Y)\xi \\ 
&& -  R^{\perp}(X,Y)\nabla_{Z}^{\perp}\xi \ = \ 0,
\end{eqnarray*}
for all $X,Y\in TM$. Conversely, let $\xi\in L^\perp$. Since $R^{\perp}(X,Y)\xi=0$ for all $X,Y \in TM$, we obtain from $(i)$ that
\[
R^{\perp}(X,Y)\nabla_{Z}^{\perp}\xi=0,
\]
for every $X,Y,Z \in TM$. Using the Ricci equation again, we obtain that
\[
\left[A_{\nabla_{Z}^{\perp}\xi},A_{\nabla_{W}^{\perp}\xi}\right]=0,
\]
for every $Z,W \in TM$. Hence, at any point $x\in M$, there exists an orthonormal basis $\{Z_1,\ldots,Z_n\}$ of $T_x M$ that diagonalizes simultaneously all shape operators $A_{\nabla_{Z}^{\perp}\xi}$, with $Z\in TM$. We claim that
\[
\left\langle \nabla_{Z_{k}}^{\perp}\xi,\alpha_f(Z_{i},Z_{j}\right\rangle=0,
\]
for all $1\leq i,j,k \leq n$, which implies that $\nabla_Z^{\perp}\xi\in N_{1}^{\perp}$, for all $Z\in TM$. From the choice of the basis $\{Z_1,\ldots,Z_n\}$, we have
\[
\left\langle\alpha (Z_i,Z_j ), \nabla_{Z_{k}}^{\perp}\xi\right\rangle = \left\langle A_{\nabla_{Z_{k}}^{\perp}\xi} Z_i, Z_j\right\rangle = \lambda_{ki}\delta_{ij},
\]
where $\lambda_{k_i}$ is the eigenvalue of $A_{\nabla_{Z_{k}}^{\perp}\xi}$ corresponding to $Z_i$. It follows from the Codazzi equation \eqref{eq:Codazzi2} and the fact that $\xi\in L^{\perp}\subset \{\eta\}^{\perp}$ that
\[
A^f_{\nabla_{Z_{i}}^{\perp}\xi}Z_{k}= A^f_{\nabla_{Z_{k}}^{\perp}\xi}Z_{i}.
\]
This implies that the eigenvalue $\lambda_{ki}$ of $A^f_{\nabla_{Z_{k}}^{\perp}\xi}$ corresponding to $Z_i$ vanishes unless $k=i$. Therefore,
\[
\left\langle\alpha_f(Z_i,Z_i),\nabla_{Z_{k}}^{\perp}\xi\right\rangle = \left\langle A^f_{\nabla_{Z_{k}}^{\perp}\xi}Z_i, Z_i\right\rangle = \left\langle A^f_{\nabla_{Z_{i}}^{\perp}\xi}Z_k,Z_i\right\rangle=0, \ \textrm{ if } i\neq k.
\]
Finally, the assumption $\nabla^{\perp}H\in L$ and $\xi\in L^{\perp}$ imply that $\langle\nabla^{\perp}_{Z_i}H,\xi\rangle=0$. Therefore,
\[
\left\langle\alpha_f(Z_{i},Z_{i}),\nabla_{Z_{i}}^{\perp}\xi, \right\rangle= n\left\langle H,\nabla_{Z_i}^{\perp}\xi\right\rangle = -n\langle\nabla^{\perp}_{Z_i}H,\xi\rangle = 0,
\]
and this concludes the proof.
\end{proof}

\section{Proof of main result }

The proof of Theorem \ref{teo:main} makes use of the codimension reduction theorem, seen in Section 1.3, and also of an important fact regarding analytical functions, as discussed below.

\begin{rem} \label{rem:analytic}
Following the approach for submanifolds in space forms, we say that a surface $M^2$ immersed in a warped product $\WP$ is real analytic if, when viewed as a graph in local coordinates, its height function is real analytic. Equivalently, if the functions of two real variables that define locally the immersion are real analytic functions. Analyticity is related to elliptic PDE's. More precisely, the geometric assumption of the isometric immersion $f\co M^2\to\WP$ has parallel mean curvature vector field $H$, combined with the Gauss, Codazzi, and Ricci equations, forms a closed system of differential equations involving the components of the immersion. Since the surface $M^2$ and the ambient $\WP$ are analytic manifolds, the the system of PDE's derived from the fundamental equations of submanifold theory is an analytic system (cf. \cite{Mo}; see also \cite{Bl}). Solutions to analytic nonlinear elliptic systems of PDE's are real analytic, and this regularity property follows from the principle of analytic continuation. 
\end{rem}

\begin{proof}[Proof of Theorem \ref{teo:main}]
Since the mean curvature vector $H$ is parallel and nonzero, the function
\[
\mu:=\|H\|^2
\]
in $M^2$ is a non-zero constant. Let us first suppose that $A_{H}=\mu I$ everywhere on $M^2$. We claim that the vector field $T$ vanishes identically. Assuming otherwise, there exists an open subset $U$, where $T\neq0$. Choose a unit vector field $X$ on $U$ orthogonal to $T$. Then
\begin{equation} \label{eq:teoACT1}
\pl\alpha_f(X,T),H\pr = \langle A_{H} X,T\rangle 
= \mu\langle X,T\rangle=0. 
\end{equation}
By the Codazzi equation \eqref{eq:Codazzi} we have 
\begin{equation} \label{eq:teoACT2}
\left\pl(\nabla^{\perp}_{T} \alpha)(X,X)- (\nabla^{\perp}_{X} \alpha)(T,X),H\right\pr = -\e\mu(t)\|T\|^2\pl\eta,H\pr.
\end{equation}
It follows from \eqref{eq:teoACT1} and the fact that $\mu$ is constant on $M^2$ that the left-hand-side of \eqref{eq:teoACT2} is zero. Thus, the function $\pl\eta,H\pr$ vanishes on $U$, and hence
\[
0 = T\langle \eta, H\rangle = \langle \nabla_{T}^{\perp} \eta, H\rangle = \left\pl-\alpha(T,T)-\rho'/\rho\langle T, T\rangle \eta,H\right\pr = -\mu \vert \vert T\vert \vert^2.
\]
Since $\mu$ is a non-zero constant, it follows that $T$ vanishes on $U$. This is a contradiction and proves the claim. Therefore, if $A_H=\mu I$ everywhere on $M^2$, then $f(M^2)$ is contained in a slice $\{t_0\}\times_{\rho}\Q^n_\e$ of  $\WP$ and either possibilities $(i)$ or $(ii)$ holds by \cite[Theorem 4]{Yau}. Assume now that $A_{H}\neq \mu I$ on an open subset $V$ of $M^2$. Since $H$ is parallel in the normal connection of $f$, it follows from the Ricci equation that $[A_H,A_\zeta]=0$, for any point $x\in M^2$ and every normal vector $\zeta\in T_xM^\perp_f$. Then the fact that $A_H$ has distinct eigenvalues on $V$ implies that the eigenvectors of $A_H$ are also eigenvectors of $A_\zeta$ for any $\zeta\in T_xM_f^\perp$, with $x\in V$. Hence all shape operators are simultaneously diagonalizable at any $x\in V$, which implies that $f$ has flat normal bundle on $V$ by the Ricci equation \eqref{eq:Ricci}. In particular, the first normal spaces $N_1$ of $f$ have dimension at most two at any point $x\in V$. Let $W\subset V$ be an open subset, where $L=N_1+span\{\eta\}$ has constant dimension $l\leq3$. It follows from Lemma \ref{lem:RedCodim} and Theorem \ref{teo:redCodim} that $f(W)$ lies in a totally geodesic submanifold $\R\times_{\rho}\Q^{2+l- 1}_{\epsilon}$ of $\R\times_{\rho}\Q^{n}_{\epsilon}$. By analyticity of $f$ (see Remark \ref{rem:analytic}), we conclude that $f(M^2)\subset\R\times_{\rho}\Q^{l+1}_\e$.
\end{proof}

\bibliographystyle{amsplain}

\begin{thebibliography}{99}

\bibitem{AbRo} U. Abresch, H. Rosenberg,\ {\em A Hopf differential for constant mean curvature surfaces in $\Sp^2\times\R$ and $\Hy^2\times\R$.} Acta Math. {\bf 193} (2004), no. 2, 141--174.

\bibitem{AlCaTr} H. Alencar, M. P. do Carmo, R. Tribuzy,\ {\em A Hopf theorem for ambient spaces of dimensions higher than three.} J. Differential Geom. {\bf 84} (2010), no. 1, 1--17.

\bibitem{AlDa} L. J. Alías, M. Dajczer,\ {\em Uniqueness of constant mean curvature surfaces properly immersed in a slab.} Comment. Math. Helv. {\bf 81} (2006), no. 3, 653--663.

\bibitem{Bl} S. Blatt,\ {\em On the analyticity of solutions to non-linear elliptic partial differential equations.} Available at arXiv:2009.08762.

\bibitem{Chen} B.-Y. Chen,\ {\em On the surface with parallel mean curvature vector.} Indiana Univ. Math. J. {\bf 22} (1972/73), 655--666.

\bibitem{Ch} S.S. Chern,\ {\em On surfaces of constant mean curvature in a three-dimensional space of constant curvature.} Lecture Notes in Math., 1007, Springer-Verlag, Berlin, 1983.

\bibitem{Da} M. Dajczer,\ {\em Reduction of the codimension of regular isometric immersions.} Math. Z. {\bf 179} (1982), no. 2, 263--286.

\bibitem{FiVi} C. do Rei Filho, F. Vitório,\ {\em A Bonnet theorem for submanifolds into rotational hypersurfaces.} Results Math. {\bf 71} (2017), no. 1-2, 283--294.

\bibitem{Er} J. Erbacher,\ {\em Reduction of the codimension of an isometric immersion.} J. Differential Geometry {\bf 5} (1971), 333--340.

\bibitem{Ho} H. Hopf,\ {\em Differential geometry in the large.} Lecture Notes in Math., 1000, Springer-Verlag, Berlin, 1983.

\bibitem{MaTo} F. Manfio, R. Tojeiro,\ {\em Hypersurfaces with constant sectional curvature in $\Sp^n\times\R$ and $\Hy^n\times\R$}, Illinois J. Math., {\bf 55} (2011), no. 1, 397--415.

\bibitem{MeTo} B. Mendon\c ca, R. Tojeiro,\ {\em Umbilical submanifolds of $\Sp^n\times\R$.} Canad. J. Math. {\bf 66} (2014), no. 2, 400--428.

\bibitem{Mo} C. B. Jr. Morrey,\ {\em On the analyticity of the solutions of analytic non-linear elliptic systems of partial differential equations. I. Analyticity in the interior.} Amer. J. Math. {\bf 80} (1958), 198--218.


\bibitem{Yau} S. T. Yau,\ {\em Submanifolds with constant mean curvature.} I. Amer. J. Math. {\bf 96} (1974), 346--366.

\end{thebibliography}

\Addresses

\end{document}